\documentclass{amsart}

\usepackage{amssymb}
\usepackage{amsmath}
\usepackage{verbatim}
\usepackage{url}
\usepackage{graphicx}
\usepackage{color}
\usepackage{enumerate}

\usepackage{tikz-cd}
\usepackage{wrapfig}
\usepackage{multicol}
\usepackage{mathrsfs}

\newcommand{\eps}{\varepsilon}

\newcommand{\opnm}{\operatorname}

\newcommand{\arctanh}{\operatorname{arctanh}}
\newcommand\xqed[1]{%
  \leavevmode\unskip\penalty9999 \hbox{}\nobreak\hfill
  \quad\hbox{#1}}

\newtheorem{theorem}{Theorem}
\newtheorem{proposition}[theorem]{Proposition}
\newtheorem{corollary}[theorem]{Corollary}
\newtheorem{lemma}[theorem]{Lemma}
\theoremstyle{definition}
\newtheorem{definition}[theorem]{Definition}
\theoremstyle{remark}
\newtheorem{remark}[theorem]{Remark}
\newtheorem{xexample}[theorem]{Example}
\newenvironment{example}{\begin{xexample}}{\xqed{$\triangle$}\end{xexample}}
\newtheorem*{acknowledgements}{Acknowledgements}

\numberwithin{equation}{section}
\numberwithin{theorem}{section}
\numberwithin{figure}{section}

\title[Norm of NSAHO semigroup]{The norm of the non-self-adjoint harmonic oscillator semigroup}

\author{Joe Viola}

\email{Joseph.Viola@univ-nantes.fr}

\address{Laboratoire de Math\'{e}matiques Jean Leray \\ 2 rue de la Houssini\`{e}re \\ Universit\'{e} de Nantes \\ BP 92208 F-44322 Nantes Cedex 3}

\begin{document}

\begin{abstract}
We identify the norm of the semigroup generated by the non-self-adjoint harmonic oscillator acting on $L^2(\Bbb{R})$, for all complex times where it is bounded. We relate this problem to embeddings between Gaussian-weighted spaces of holomorphic functions, and we show that the same technique applies, in any dimension, to the semigroup $e^{-tQ}$ generated by an elliptic quadratic operator acting on $L^2(\Bbb{R}^n)$. The method used --- identifying the exponents of sharp products of Mehler formulas --- is elementary and is inspired by more general works of L.\ H\"ormander, A.\ Melin, and J.\ Sj\"ostrand.
\end{abstract}

\maketitle

\section{Introduction}

We consider the non-self-adjoint harmonic oscillator, often called the Davies operator after the contributions of E.\ Brian Davies,
\begin{equation}\label{eq_def_NSAHO}
	Q_\theta = -e^{-i\theta}\frac{d^2}{dx^2} + e^{i\theta} x^2
\end{equation}
acting on $L^2(\Bbb{R})$ for $\theta \in (-\frac{\pi}{2}, \frac{\pi}{2})$. We refer the reader to \cite[Section 14.5]{Davies_2007} or \cite[Section VII.D]{Krejcirik_Siegl_Tater_Viola_2014} for an introduction and a summary of recent results. As a (necessarily non-exhaustive) list of articles related to this family of operators, we mention \cite{Exner_1983, Davies_1999a, Davies_1999b, Boulton_2002, Davies_Kuijlaars_2004, Swanson_2004, Pravda-Starov_2006, Bordeaux-Montrieux_2013, Henry_2014, Graefe_Korsch_Rush_Schubert_2015}.

Even though $Q_\theta$ has a compact resolvent and simple real eigenvalues $\{1+2k\,:\, k\in\Bbb{N}\}$, the operator $e^{-tQ_\theta}$, realized as the graph closure starting on the span of its eigenfunctions, is only bounded on $L^2(\Bbb{R})$ when $t \in \overline{\Omega_\theta}$ with
\begin{equation}\label{eq_def_Omega_theta}
	\Omega_\theta = \{\Re t > 0\} \cap \left\{t \in \Bbb{C} \::\: |\arg \tanh t| \leq \frac{\pi}{2} - |\theta| \right\}.
\end{equation}
(This follows from \cite[Section 1.2.1]{Aleman_Viola_2014b}; see proposition \ref{prop_bddness_same}.) Note that $\overline{\Omega_0} = \{\Re t \geq 0\}$, agreeing with the set where $e^{-tQ_0}$ is bounded, and if $\theta \in (-\frac{\pi}{2}, \frac{\pi}{2})$ is nonzero, then $\overline{\Omega_\theta} = \Omega_\theta \cup \frac{i\pi}{2}\Bbb{Z}$.

The purpose of this article is to identify the norm of $e^{-tQ_\theta}$. The problem is trivial for $t = i \pi k \in i\pi \Bbb{Z}$, since then $e^{-i\pi k Q_\theta} = (-1)^k$ is unitary. The same is true for all $t \in \frac{i\pi}{2}\Bbb{Z}$ since, by a parity argument, $e^{-\frac{i\pi}{2}Q_\theta}u(x) = -iu(-x)$; see remark \ref{rem_bdry_Omega}.  Otherwise, for all $t \in \Bbb{C}\backslash i\Bbb{R}$ for which $e^{-tQ_\theta}$ is bounded on $L^2(\Bbb{R})$, we obtain the following formula.

\begin{theorem}\label{thm_NSAHO_norm}
Let $Q_\theta$ be as in \eqref{eq_def_NSAHO} with $|\theta| < \pi/2$, and let $t \in \Omega_\theta$. Write $\phi = \arg \tanh t \in (-\pi, \pi)$ and
\begin{equation}\label{eq_def_Acste}
	A = \frac{1}{2}|\sinh 2t|^2\left(\cos 2\theta + \cos 2\phi\right).
\end{equation}
Then, as an operator in $\mathcal{L}(L^2(\Bbb{R}))$,
\begin{equation}\label{eq_NSAHO_norm}
	\|e^{-tQ_\theta}\| = \left(\sqrt{1+A} + \sqrt{A}\right)^{-1/2}.
\end{equation}
\end{theorem}

\begin{figure}
  \centering
    \includegraphics[width=0.3\textwidth]{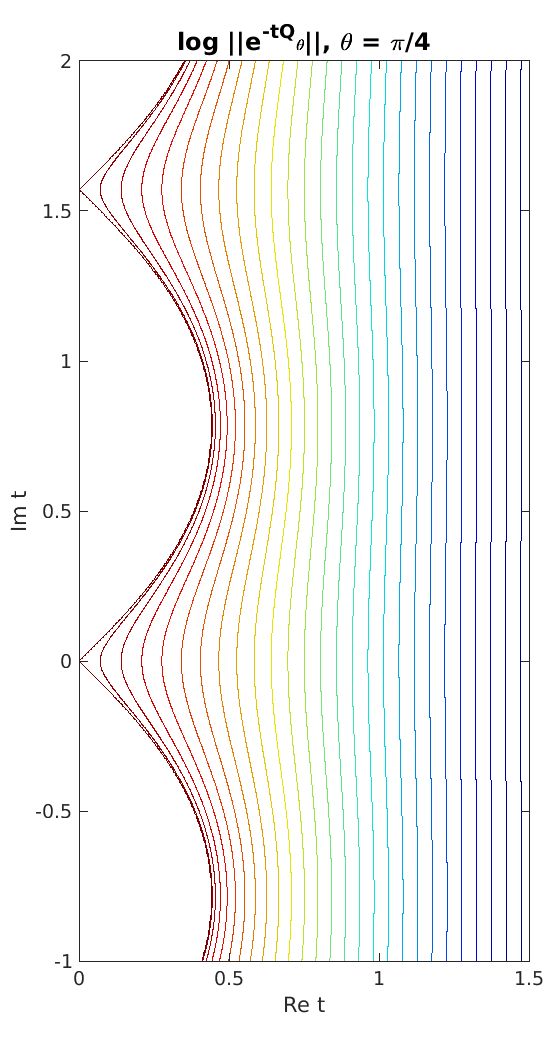}
	\includegraphics[width=0.3\textwidth]{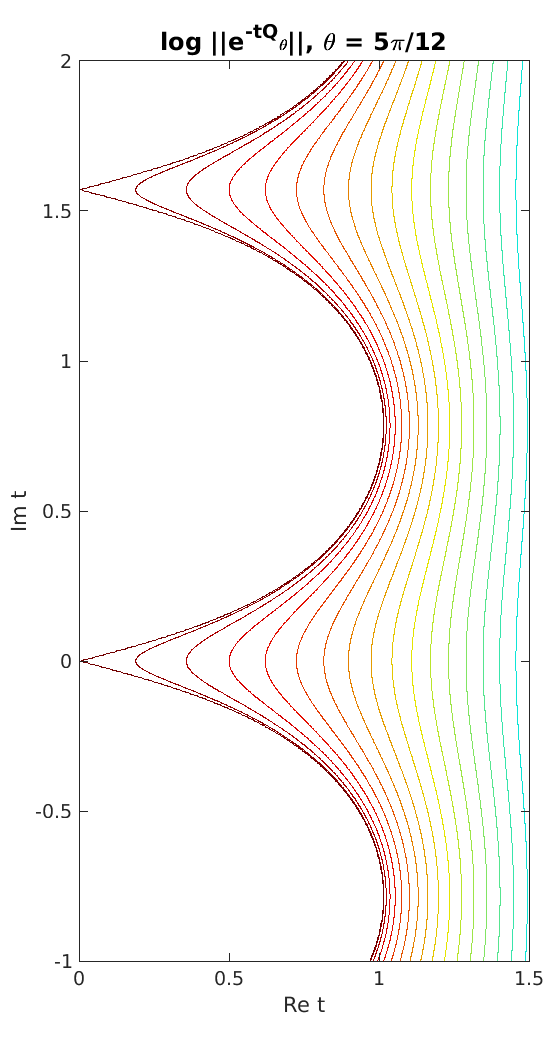}
  \caption{$\log\|e^{-tQ_\theta}\|$ for $\theta = \frac{\pi}{4}, \frac{5\pi}{12}$ and contour lines $0, -0.05, -0.1, \dots, -1.3$.}
	\label{f_Contour}
\end{figure}

A related problem concerns embeddings between spaces of holomorphic functions with Gaussian weights.  Let the weight $\Phi:\Bbb{C} \to \Bbb{R}$ be real-quadratic and strictly plurisubharmonic, meaning that
\begin{equation}\label{eq_stplsh}
	\partial_{\bar{z}}\partial_z \Phi = \frac{1}{4}(\partial_{\Re z}^2 + \partial_{\Im z}^2)\Phi > 0.
\end{equation}
(Derivatives with respect to the complex variable $z$ are assumed throughout to be holomorphic.) Then we define
\begin{equation}\label{eq_def_HPhi}
	H_\Phi = \opnm{Hol}(\Bbb{C}) \cap L^2(\Bbb{C}, e^{-2\Phi(z)}\,d\Re z \,d \Im z),
\end{equation}
the space of holomorphic functions with finite weighted-$L^2$ norm
\[
	\|u\|_\Phi = \left(\int_{\Bbb{C}} |u(z)|^2 e^{-2\Phi(z)}\,d\Re z \,d \Im z\right)^{1/2}.
\]
Given any two such weights $\Phi_1$ and $\Phi_2$, we consider the norm of the embedding
\[
	H_{\Phi_1} \ni u \stackrel{\iota}{\mapsto} u \in H_{\Phi_2},
\]
which is bounded if and only if $\Phi_1 \leq \Phi_2$ (see e.g.\ \cite[Proposition 2.4, Corollary 2.6]{Aleman_Viola_2014b}).

Finding this norm is equivalent to finding the norm of $e^{-tQ_\theta}$ in view of proposition \ref{prop_NSAHO_FBI_side}. In addition, $\|\iota\|$ represents a sort of anisotropic uncertainty principle for holomorphic functions: clearly $\|\iota\| \leq 1$, and if functions in $H_{\Phi_1}$ could concentrate arbitrarily closely to the origin, then $\|\iota\|$ would equal $1$. When $\partial_z^2 \Phi_j = 0$ for $j=1,2$ (in any dimension), the connection between the norm of the embedding and the uncertainty principle (in the sense of the minimum of the spectrum of the harmonic oscillator $Q_0$) is discussed in \cite[Section 4.3]{Aleman_Viola_2014b}.

\begin{theorem}\label{thm_HPhi_embedding}
Let $\Phi_1$ and $\Phi_2$ be two real-valued real-quadratic forms which are strictly plurisubharmonic as in \eqref{eq_stplsh}. Define
\[
	a = \frac{\partial_{\bar{z}}\partial_z \Phi_2}{\partial_{\bar{z}}\partial_z \Phi_1}, \quad b = \frac{|\partial_z^2(\Phi_2-\Phi_1)|}{\partial_{\bar{z}}\partial_z\Phi_1}.
\]
Let $\iota$ be the embedding from $H_{\Phi_1}$ into $H_{\Phi_2}$, which is bounded if and only if $a -b \geq 1$, that is to say, if and only if $\Phi_1 \leq \Phi_2$.

If $b = 0$, then
\[
	\|\iota\| = a^{-1/2}.
\]

If $a-b = 1$ and $b > 0$, then
\[
	\|\iota\| = a^{-1/4}.
\]

Finally, if $a-b>1$, then for
\[
	\gamma = \frac{1}{2b}\left(1 - a^2 + b^2 + \sqrt{(a^2 - b^2 - 1)^2 - 4b^2}\right),
\]
\[
	\|\iota\| = \left(\frac{1-\gamma^2}{a^2 - (b - \gamma)^2}\right)^{1/4}.
\]
\end{theorem}

While the results above focus on the non-self-adjoint harmonic oscillator in one dimension, the same analysis allows us to determine the norm of $e^{-tQ}$ as an operator in $L^2(\Bbb{R}^n)$ when $t > 0$ and $Q(x,D_x)$ is any elliptic quadratic operator, regardless of the dimension $n$; see in particular corollary \ref{cor_higher_dim_Mehler} below.

The strategy followed here is straightforward and takes its inspiration from deeper works like \cite{Hormander_1983, Hormander_1995}. In the case of theorem \ref{thm_NSAHO_norm}, when $e^{-tQ_\theta}$ is compact, the sharp product of its Mehler formula and that of its adjoint must give, up to symplectic equivalence and a factor depending on $t$ and $\theta$, the Mehler formula coming from a harmonic oscillator. We identify this harmonic oscillator and obtain the norm as an immediate consequence. While the formulas obtained are explicit, it seems clear that a deeper analysis could be performed to understand results obtained through sometimes opaque direct computations, particularly in higher dimensions.

The plan of the paper is as follows.  In the following section, we introduce the standard tools of the Weyl quantization and Mehler formulas and establish some simple results around the harmonic oscillator semigroup and Mehler formulas of real quadratic type. Next, in section \ref{sec_NSAHO}, we perform the computations leading to theorem \ref{thm_NSAHO_norm}. In section \ref{sec_FBI}, we explain the relationship between the non-self-adjoint harmonic oscillator semigroup and embeddings between Gaussian-weighted spaces by recalling the dimension-one FBI--Bargmann theory and we prove theorem \ref{thm_HPhi_embedding}. Finally, in section \ref{sec_higher_dimension}, we discuss the natural extension of these results to dimensions $n > 1$.

\begin{acknowledgements} The author would like to thank Johannes Sj\"ostrand for helpful advice, including pointing out reference \cite{Hormander_1983}, and Bernard Helffer for many interesting and illuminating discussions. The author is also indebted to the anonymous referee for many helpful suggestions and improvements to the present work. The author is also grateful for the support of the Agence Nationale de la Recherche (ANR) project NOSEVOL, ANR 2011 BS01019 01.
\end{acknowledgements}

\section{The Weyl quantization, Mehler formulas, and their sharp products}

We briefly recall the essential tools of the Weyl quantization, composition of symbols, and the quantization of the linear symplectic group. Being primarily interested in dimension one and symbols which are either polynomials or bounded with all derivatives, we only present a sketch of the general theory, which may be found in, for instance, \cite[Sections 18, 21]{Hormander_ALPDO_3}.

When $a\in \mathscr{S}'(\Bbb{R}^{2n})$, the Weyl quantization may be defined weakly for $u,v \in \mathscr{S}(\Bbb{R}^n)$ as
\begin{equation}\label{eq_def_Weyl}
	\langle a^w(x,D_x)u, v\rangle = \frac{1}{(2\pi)^n} \int e^{i(x-y)\cdot \xi} a\left(\frac{x+y}{2}, \xi\right) u(y)\overline{v(x)}\,dy\,d\xi\,dx.
\end{equation}
If $a$ is bounded with all derivatives, the operator $a^w(x,D_x)$ then has a continuous extension to a bounded operator on $L^2(\Bbb{R}^n)$ by the Calderon-Vaillancourt Theorem. If $a(x,\xi)$ is a polynomial in $(x, \xi)$, the standard computations with the Fourier transform give $a^w(x,D_x)$ as a polynomial in the non-commuting operators of multiplication by $x_j$ and $D_{x_j} = -i(\partial/\partial x_j)$, for $j=1,\dots,n$.

\begin{example}
If $q(x,\xi) = ax^2 + 2bx\xi + c\xi^2$ for $(x,\xi) \in \Bbb{R}^2$, then
\[
	q^w(x,D_x) = ax^2 + \frac{b}{i}\left(x\frac{d}{dx} + \frac{d}{dx}x\right) - c\frac{d^2}{dx^2}.
\]
In particular, the Weyl symbol of $Q_\theta$ in \eqref{eq_def_NSAHO} is
\begin{equation}\label{eq_NSAHO_symbol}
	q_\theta(x,\xi) = e^{-i\theta}\xi^2 + e^{i\theta}x^2.
\end{equation}
\end{example}

The standard symplectic form on $T^*\Bbb{R}^{n} \approx \Bbb{R}^n_x \times \Bbb{R}^n_\xi$ is
\[
	\sigma((x,\xi),(y,\eta)) = \xi\cdot y - \eta\cdot x.
\]
A (real) linear transformation $K$ is said to be canonical if, for all $(x,\xi), (y,\eta) \in \Bbb{R}^{2n}$,
\[
	\sigma(K(x,\xi), K(y,\eta)) = \sigma((x,\xi),(y,\eta)).
\]
In dimension one, this is equivalent to assuming that $\det K = 1$.

\begin{definition}\label{def_metaplectic} A continuous automorphism $\mathcal{U}$ of the Schwartz space $\mathscr{S}'(\Bbb{R}^n)$ is a member of the metaplectic group \cite[Definition 1.4]{Leray_1981} if the conjugation $L \mapsto \mathcal{U}L\mathcal{U}^{-1}$ preserves the set
\[
	\{\ell^w(x,D_x) \::\: \ell : \Bbb{R}^{2n}\to \Bbb{R}\textnormal{ is linear}\}
\]
and if the restriction of $\mathcal{U}$ to $L^2(\Bbb{R}^n)$ is unitary.

More concretely, $\mathcal{U}$ belongs to the metaplectic group if and only if it may be written as a composition of linear changes of variables, multiplication by imaginary Gaussians, and the Fourier transform in one variable.

In this case, there is a real linear canonical transformation $K$ associated with $\mathcal{U}$ for which
\begin{equation}\label{eq_metaplectic_Egorov}
	\mathcal{U}a^w(x,D_x)\mathcal{U}^* = (a\circ K)^w(x,D_x), \quad \forall a \in \mathscr{S}'(\Bbb{R}^{2n}).
\end{equation}
This may be reversed \cite[Theorem 18.5.9]{Hormander_ALPDO_3}: for any real linear canonical transformation $K$ on $\Bbb{R}^{2n}$, there exists an element of the metaplectic group for which \eqref{eq_metaplectic_Egorov} holds, and this element is determined by $K$ uniquely up to a factor of modulus one.
\end{definition}

We say that the symbols $a(x,\xi)$ and $b(x,\xi)$ are symplectically equivalent if there exists some real linear canonical transformation $K$ such that $a\circ K = b$. This implies that their Weyl quantizations are unitarily equivalent as operators on $L^2(\Bbb{R}^n)$.

Any positive semidefinite quadratic form on $\Bbb{R}^2$ is symplectically equivalent to either $x^2$ or $\rho(x^2 + \xi^2)$ for some $\rho \geq 0$; see \cite[Theorem 21.5.3]{Hormander_ALPDO_3}.

Any complex-valued quadratic form on $\Bbb{R}^2$ for which $q^{-1}(\{0\}) = \{(0,0)\}$ and $q(\Bbb{R}^2) \neq \Bbb{C}$, a condition which is satisfied if $q$ has positive definite real part, is symplectically equivalent to $\rho q_\theta(x,\xi)$ for some $\theta \in [0, \pi/2)$ and with $\rho \in \Bbb{C}\backslash\{0\}$; see \cite[Section 7.8]{Krejcirik_Siegl_Tater_Viola_2014} and references therein, particularly to \cite{Pravda-Starov_2007}.

The symplectic polarization of a quadratic form $q(x,\xi):\Bbb{R}^2 \to \Bbb{C}$ is made through the fundamental matrix
\[
	 F = \frac{1}{2}\left(\begin{array}{cc} q''_{\xi x} & q''_{\xi\xi} \\ -q''_{xx} & -q''_{x \xi}\end{array}\right),
\]
the unique matrix antisymmetric with respect to $\sigma$ for which
\[
	q(x,\xi) = \sigma((x,\xi),F(x,\xi)), \quad \forall (x,\xi) \in \Bbb{R}^{2n}.
\]
Since $\sigma(X, Y) = X\cdot JY$ for all $X,Y \in \Bbb{R}^{2n}$ and $J \in GL_{2n}(\Bbb{C})$ given by
\begin{equation}\label{eq_def_J}
	J = \left(\begin{array}{cc} 0 & -1 \\ 1 & 0\end{array}\right),
\end{equation}
note that
\[
	F = -\frac{1}{2}J\opnm{Hess}(q)
\]
for $\opnm{Hess}(q) = \nabla^2_{x,\xi}q$ the Hessian matrix of second derivatives of $q$.

\begin{example}\label{ex_HO_reduce}
We recall how to use the fundamental matrix to deduce the harmonic oscillator structure of a positive definite quadratic operator in dimension one. Let
\[
	q(x,\xi) = ax^2 + 2bx\xi + c\xi^2
\]
be positive in the sense that $q(x,\xi) > 0$ for any $(x,\xi) \in \Bbb{R}^2\backslash \{0\}$. (This implies that the coefficients are real.) We can easily check that
\[
	\det(F-\lambda) = \lambda^2 - b^2 + ac = \lambda^2 + \det F .
\]
Note that $\det F$ must be positive by positivity of $q$. Writing $\delta = \det F$ gives $\opnm{Spec}F = \pm i\sqrt{\delta}$, and the eigenspaces of $F$ are
\[
	\ker (F - i\sqrt{\delta}) = \{(x,i\gamma x)\}_{x\in\Bbb{C}}, \quad \ker(F+i\sqrt{\delta}) = \{(x,-i\overline{\gamma}x)\}_{x\in\Bbb{C}}
\]
where $\gamma = \frac{1}{c}(\sqrt{\delta} + ib)$.

These linear algebra facts are in correspondence with the expression
\[
	q(x,\xi) = \frac{\sqrt{\delta}}{\Re \gamma}|\xi - i\gamma x|^2,
\]
from which we can check that
\[
	q^w(x,D_x)e^{-\gamma x^2/2} = \sqrt{\delta}e^{-\gamma x^2/2}
\]
and that $q(x,\xi)$ is symplectically equivalent to $\sqrt{\delta}(x^2 + \xi^2)$.
\end{example}

The fundamental matrix allows us to write the Mehler formula \cite[Theorem 4.2]{Hormander_1995}
\begin{equation}\label{eq_Mehler}
	M_q(x,\xi) = \frac{1}{\sqrt{\det\cos F}}\exp(\sigma((x,\xi), (\tan F)(x,\xi))
\end{equation}
for which, at least where $\Re q \leq 0$,
\[
	\exp(q^w(x, D_x)) = M^w_q(x,D_x).
\]

\begin{example}\label{ex_NSAHO_Mehler}
For the symbol $q_\theta(x,\xi)$ in \eqref{eq_NSAHO_symbol},
\[
	F = \left(\begin{array}{cc} 0 & e^{-i\theta} \\ -e^{i\theta} & 0\end{array}\right).
\]
Since $F^2 = -I$ for any $\theta \in \Bbb{R}$, the corresponding Mehler formula is
\[
	M_{-tq_\theta}(x,\xi) = \frac{1}{\cosh t}\exp(-\tanh(t) q_\theta(x,\xi)).
\]

We note that $\Omega_\theta$ in \eqref{eq_def_Omega_theta} is the preimage under hyperbolic tangent of the sector
\[
	\{t \in \Bbb{C}\backslash \{0\} \::\: \Re (tq_\theta(x,\xi)) \geq 0 \quad \forall (x,\xi) \in \Bbb{R}^2\}
\]
and coincides with the set of $t \in \Bbb{C}\backslash i\Bbb{R}$ such that $M_{-tq_\theta}(x,\xi)$ is bounded on $\Bbb{R}^2$.
\end{example}

We also recall that, for two symbols $a(x,\xi), b(x,\xi)$ which are bounded with all derivatives,
\[
	a^w(x,D_x)b^w(x,D_x) = (a\sharp b)^w(x,D_x)
\]
with
\begin{equation}\label{eq_sharp_product}
	(a\sharp b)(x,\xi) = \left.e^{\frac{i}{2}\sigma((D_x, D_\xi), (D_y, D_\eta))}a(x,\xi)b(y,\eta)\right|_{(y,\eta) = (x,\xi)}
\end{equation}
This is a pseudodifferential operator on $\Bbb{R}^{4n}$, and we may directly compute its action on Gaussian symbols.

We will see in proposition \ref{prop_sharp_higher_dim} that the sharp product of Gaussian symbols gives a Gaussian symbol. Consider a symbol 
\[
	a(x,\xi) = e^{-q(x,\xi)}, \quad (x,\xi) \in \Bbb{R}^2,
\]
where $q$ is quadratic and $\Re q \geq 0$. If $a^w(x,D_x)$ is self-adjoint, then $q$ is real-valued. In this case, $q(x,\xi)$ is therefore symplectically equivalent to $0$, to $x^2$, or to $r(x^2 + \xi^2)$ for some $r > 0$, according to whether $a^w(x,D_x)$ is the identity, a bounded but not compact operator, or a compact operator.

Since $\tanh:[0,\infty) \to [0,1)$, the Mehler formula for the harmonic oscillator only allows us to analyze $e^{-r(x^2 + \xi^2)}$ when $r \in [0,1)$. When $r=1$, the corresponding operator is (up to a constant) orthogonal projection onto the Gaussian $e^{-x^2/2}$, reflecting that it is the limit of the Mehler formula for $e^{-tQ_0}$ as $t \to \infty$ (see \cite[Section 1.4]{Helffer_1984}). Beyond, when $r > 1$, holomorphic continuation gives a non-positive compact operator.

\begin{proposition}\label{prop_HO_Mehler_all_times}
Let $a_r(x,\xi) = e^{-r(x^2 + \xi^2)}$ for some $r \in \Bbb{C}$ with $\Re r > 0$. Then
\begin{equation}\label{eq_HO_all_times}
	a_r^w(x,D_x) = \cosh(\arctanh r)e^{-(\arctanh r)Q_0}
\end{equation}
in the sense that, for $h_k(x)\in \ker(Q_0 - (2k+1))$ the Hermite functions,
\begin{equation}\label{eq_HO_all_times_Hermites}
	a_r^w(x,D_x)h_k(x) = \frac{(1-r)^k}{(1+r)^{k+1}}h_k(x).
\end{equation}
As a consequence, $\|a_r^w(x,D_x)\| = |1+r|^{-1}$.
\end{proposition}

\begin{proof}
The formula for $a_r^w(x,D_x)h_k(x)$ is clearly holomorphic in $r$ for $\Re r > 0$, and by the Mehler formula for the harmonic oscillator, for $r \in (0,1)$,
\[
	a_r^w(x,D_x)h_k(x) = \cosh(\arctanh r)e^{-(2k+1)\arctanh r}h_k(x) = \frac{(1-r)^k}{(1+r)^{k+1}}h_k(x)
\]
by usual formulas like $e^{-\arctanh r} = \sqrt{\frac{1-r}{1+r}}$ when $r \in (0,1)$. The formula \eqref{eq_HO_all_times_Hermites} is holomorphic on $\{\Re r > 0\}$ and therefore coincides with $a^w_r(x,D_x)h_k(x)$ on all of $\{\Re r > 0\}$.

The norm of $a_r^w(x,D_x)$ is then attained on the function $h_0(x)$, since the $h_k$ are orthogonal and $|1-r| < |1+r|$ on $\{\Re r > 0\}$.
\end{proof}

\begin{corollary}\label{cor_positive_compact}
Let $a(x,\xi) = e^{-q(x,\xi)}$ with $q:\Bbb{R}^2 \to \Bbb{C}$ quadratic. If $a^w(x,D_x)$ is self-adjoint, positive, compact, and is not rank 1, then $q(x,\xi)$ is symplectically equivalent to $r(x^2 + \xi^2)$ for some $r \in (0,1)$. If $a^w(x,D_x)$ is self-adjoint and bounded but not compact, then either $q = 0$ or $q(x,\xi)$ is symplectically equivalent to $x^2$.
\end{corollary}

We finish with a straightforward observation regarding our ability to pick out functions of harmonic oscillators via their norms on Gaussians
\begin{equation}\label{eq_u_gamma_def}
	u_\gamma(x) = \left(\frac{\pi}{\Re \gamma}\right)^{-1/4} e^{-\frac{\gamma}{2} x^2}
\end{equation}
for $\gamma \in \Bbb{C}_+,$
\begin{equation}\label{eq_Cplus_def}
	\Bbb{C}_+ = \{\gamma \in \Bbb{C} \::\: \Re \gamma > 0\}.
\end{equation}
This observation is motivated by $f(t) = e^{-t}$, for use in section \ref{ssec_norm_FBI}.

\begin{proposition}\label{prop_gaussians}
Let $Q = q^w(x,D_x)$ for $q:\Bbb{R}^2\to [0,\infty)$ quadratic, real-valued, and positive semidefinite. Let $f:[0, \infty) \to (0, \infty)$ be a strictly decreasing Borel function, so that $f(Q)$ may be defined by the functional calculus. Then
\[
	\|f(Q)\| = \sup_{\gamma \in \Bbb{C}_+} \|f(Q)u_\gamma\|_{L^2(\Bbb{R})}.
\]
Furthermore, we may identify the form of $Q$ by examining where this norm is attained in the following way.
\begin{enumerate}[(i)]
\item\label{it_gaussian_HO} The inequality $\|f(Q)\| < f(0)$ holds if and only if $Q$ is unitarily equivalent to $\lambda Q_0$, where $\lambda = f^{-1}(\|f(Q)\|)$. In this case, there exists a unique $\gamma_0 \in \Bbb{C}_+$ for which $\|f(Q)\| = \|f(Q)u_{\gamma_0}\|$. Consequently, $q(x,\xi) = \frac{\lambda}{\Re \gamma_0}|\xi - i\gamma_0 x|^2$.
\item\label{it_gaussian_zero} If $\|f(Q)\| = f(0)$, then $\|f(Q)u_\gamma\| = f(0)$ for some $\gamma \in \Bbb{C}_+$ if and only if $\|f(Q)u_\gamma\| = f(0)$ for all $\gamma \in \Bbb{C}_+$ if and only if $Q = 0$.
\item\label{it_gaussian_heat} If $\|f(Q)\| = f(0)$ but $\|f(Q)u_\gamma\| \neq f(0)$ for some (all) $\gamma \in \Bbb{C}_+$, then there exists a unique $\gamma_0 \in \partial \Bbb{C}_+ \cup \{+\infty\}$ such that
\[
	f(0) = \mathop{\lim_{\gamma \to \gamma_0}}_{\gamma\in\Bbb{C}_+} \|f(Q)u_\gamma\|,
\]
where the limit is non-tangential (if $\gamma_0 = +\infty$, the limit is over any closed sector contained in $\Bbb{C}_+$). In this case, there is some $\lambda > 0$ such that $q(x,\xi) = \lambda x^2$ if $\gamma_0 = +\infty$ or $q(x,\xi) = \lambda(\xi - i\gamma_0 x)^2$ if $\gamma_0 \in \partial \Bbb{C}_+ = i\Bbb{R}$.
\end{enumerate}
\end{proposition}

\begin{proof}
We know that $q$ is symplectically equivalent to either $0$, $x^2$, or $\lambda(x^2 + \xi^2)$ for some $\lambda > 0$; by conjugating with some element of the metaplectic group (definition \ref{def_metaplectic}), we begin by assuming that $q$ is in one of these forms. In this case, (\ref{it_gaussian_HO}) is obvious since $\opnm{Spec}\lambda Q_0 = \lambda(1+2\Bbb{N})$ and $\lambda Q_0 u_1 = \lambda u_1$. It is also easy to see (\ref{it_gaussian_heat}) because, though the norm is not attained for any $L^2$ function, for any $\eta > 0$ fixed,
\[
	\sup_{\|u\| = 1} \|f(x^2) u\| = f(0) = \mathop{\lim_{\gamma \to \infty}}_{|\arg \gamma| < \pi/2-\eta} \|f(x^2)u_\gamma\|.
\]
(We take non-tangential limits at $+\infty$ in order to exclude degenerate situations like $\gamma = 1+it$, which tends to $\infty$ as $0 < t \to \infty$.) The only case remaining is $Q = 0$, where the characterization is obvious.

The result follows by undoing the reduction of $q$ to a normal form. Suppose that
\[
	\mathcal{U}q^w(x,D_x)\mathcal{U}^* = p^w(x,D_x)
\]
with $p(x,\xi)$ either $0$, $x^2$, or $\lambda(x^2 + \xi^2)$ for some $\lambda > 0$, with $\mathcal{U}$ an element of the metaplectic group for which, for all $g \in \mathscr{S}'(\Bbb{R}^{2n})$,
\begin{equation}\label{eq_id_gauss_metapl}
	\mathcal{U}^*g^w(x,D_x)\mathcal{U} = (g\circ K^{-1})^w(x,D_x), \quad K^{-1} = \left(\begin{array}{cc} a & b \\ c & d\end{array}\right),
\end{equation}
as in \eqref{eq_metaplectic_Egorov}. By the previous discussion and the fact that 
\[
	\|f(Q)\| = \|f(\mathcal{U}^*p^w(x,D_x)\mathcal{U})\| = \|\mathcal{U}^* f(p^w(x,D_x))\mathcal{U}\| = \|f(p^w(x,D_x))\|,
\]
we can deduce the form of $p$ from $\|f(Q)\|$ as in the proposition.

As for the Gaussian for which the norm is attained, we will see that $\mathcal{U}$ induces a linear fractional (M\"obius) transformation, bijective on $\Bbb{C}_+$, on the family of Gaussians parameterized by $\gamma \in \Bbb{C}_+$. If $\gamma \in \Bbb{C}_+$, then $u_\gamma(x)$ is identified up to a coefficient of modulus one as the $L^2$-normalized element of $\opnm{ker}(iD_x + \gamma x)$. We compute that $\mathcal{U}^*u_\gamma(x)$ is a $L^2$-normalized element of the kernel of
\[
	\mathcal{U}^*(iD_x + \gamma x)\mathcal{U} = (d-i\gamma b)\left(iD_x + \frac{ic + \gamma a}{d-i\gamma b}x\right),
\]
by \eqref{eq_id_gauss_metapl}.

Therefore $K$ induces the linear fractional transformation
\begin{equation}\label{eq_lft_gaussians}
	L(\gamma) = \frac{ic+\gamma a}{d-i\gamma b}
\end{equation}
on the parameter $\gamma$ of the Gaussian $u_\gamma$. Since $\{\Re \gamma = 0\}$ is invariant under $L$ and because
\[
	\Re L(1) = \frac{\det K^{-1}}{b^2 + d^2} = \frac{1}{b^2 + d^2} > 0,
\]
we see that $L$ is an automorphism of $\Bbb{C}_+$. (One could also deduce these facts from integrability of $u_{L(\gamma)}$ and our ability to invert $K$ and $L$.)

This linear fractional transformation induced by $K$ allows us to find $\gamma_0$ in (\ref{it_gaussian_HO}) and (\ref{it_gaussian_heat}); the form of $q(x,\xi)$ follows from composition with $K$ and, in the harmonic oscillator case, the computations in example \ref{ex_HO_reduce}.

\end{proof}

\section{The non-self-adjoint harmonic oscillator}\label{sec_NSAHO}

We now compute the norm of the non-self-adjoint harmonic oscillator via the sharp product of its Mehler formula with its complex conjugate. In the boundary case $t \in \partial\Omega_\theta \backslash i\Bbb{R}$, the sharp product is not of symbols bounded with all derivatives (since for instance $|\partial_x^k e^{ix^2}|$ grows like $|x|^k$), but following \cite[Theorem 2.9]{Aleman_Viola_2014b} (see proposition \ref{prop_FBI_boundedness} below), the family is strongly continuous in $t$ and it suffices to check the formula on the eigenfunctions of $Q_\theta$ which decay superexponentially along with their Fourier transforms. A similar argument justifies the use of the Mehler formula beyond the sector of $t \in \Bbb{C}$ for which $\Re t q_\theta(x,\xi) \geq 0$.

In the process of proving this result, we actually prove the following unitary reduction, from which theorem \ref{thm_NSAHO_norm} follows immediately.

\begin{theorem}\label{thm_NSAHO_SVD}
With the assumptions and notations of theorem \ref{thm_NSAHO_norm}, again with $t \in \Omega_\theta$, let
\[
	\delta = \frac{A}{1+A} \in [0,1).
\]
Then $\delta = 0$ if and only if $|\phi| + |\theta| = \pi/2$, in which case there exists $\mathcal{U}$, an element of the metaplectic group (definition \ref{def_metaplectic}), for which
\[
	\mathcal{U}\left((e^{-tQ_\theta})^*e^{-tQ_\theta}\right)^{1/2}\mathcal{U}^* = e^{-x^2}.
\]
Otherwise, when $\delta > 0$, there is some $\mathcal{U}$, an element of the metaplectic group, for which
\[
	\mathcal{U}\left((e^{-tQ_\theta})^*e^{-tQ_\theta}\right)^{1/2}\mathcal{U}^* = e^{-\frac{1}{2}\arctanh(\sqrt{\delta})Q_0}.
\]
\end{theorem}

\begin{proof}
Write $T = \tanh t$. The symbol of $e^{-tQ_\theta}$ is $M_{-tq_\theta}(x,\xi)$ from example \ref{ex_NSAHO_Mehler}, and the symbol of $(e^{-tQ_\theta})^*$ is easily seen to be $\overline{M_{-tq_\theta}} = M_{-\overline{t}q_{-\theta}}$. We compute the symbol of $(e^{-tQ_\theta})^*e^{-tQ_\theta}$ via the sharp product \eqref{eq_sharp_product} and the Fourier transform of a Gaussian; note that $t \in \Omega_\theta$ implies that $Te^{i\theta} \neq 0$ and that $\Re (Te^{\pm i\theta}), \Re (\overline{T}e^{\pm i\theta}) \geq 0$. Writing $Z = (x,\xi,y,\eta)$ and similarly for $Z^*$ and $Z_*$,
\begin{multline*}
	e^{\frac{i}{2}\sigma((D_x, D_\xi), (D_y, D_\eta))}e^{-\bar{t}q_{-\theta}(x,\xi)}e^{-tq_\theta(y,\eta)}
\end{multline*}
\[
	= (2\pi)^{-4} \int e^{iZ^*(Z-Z_*)+\frac{i}{2}(\xi^*y^*-\eta^*x^*) - \overline{T}q_{-\theta}(x_*,\xi_*) - Tq_\theta(y_*, \eta_*)}\,dZ_*dZ^*
\]
\[
	= (2\pi)^{-4} \sqrt{\frac{\pi^4}{|T|^4}}\int e^{iZ^*Z + \frac{i}{2}(\xi^*y^*-\eta^*x^*) - \frac{1}{4\overline{T}}q_{\theta}(x^*,\xi^*) - \frac{1}{4T}q_{-\theta}(y^*,\eta^*)}\,dZ^*
\]
\[
	= (2\pi)^{-4} \sqrt{\frac{16\pi^6}{\overline{T}^2}}\int e^{i(xx^* + \xi\xi^*) - \frac{1}{4\overline{T}}q_\theta(x^*, \xi^*) - Te^{i\theta}(y+\frac{\xi^*}{2})^2 - Te^{-i\theta}(\eta - \frac{x^*}{2})^2}\,dx^*\,d\xi^*
\]
\[
	= \frac{1}{\left|1 + e^{2i\theta}|T|^2\right|} \exp\left(-Tq_\theta(y,\eta) + \frac{(ix+Te^{-i\theta}\eta)^2}{\frac{1}{\overline{T}e^{-i\theta}}+e^{-i\theta}T} + \frac{(i\xi - Te^{i\theta}y)^2}{\frac{1}{\overline{T}e^{i\theta}} + e^{i\theta}T}\right).
\]

Writing the exponent in a manageable fashion does not seem obvious. We take inspiration from the importance of 
\[
	f(z) = z+z^{-1}
\]
in simplifying Mehler formulas for the harmonic oscillator: in particular, if $\theta = 0$ and $t > 0$, putting $(y,\theta) = (x,\xi)$ gives an exponent of $-2/f(T) = -\tanh(2t)$. We also define
\[
	\phi = \arg T, \quad A_\theta = |T|e^{i\theta},
\]
so we can write
\[
	\left(\frac{1}{\overline{T}e^{-i\theta}} + e^{-i\theta} T\right)^{-1} = \frac{e^{-i\phi}}{f(A_{-\theta})}.
\]
In addition, $\overline{A_\theta} = A_{-\theta}$ and $f(z)$ is holomorphic. This is useful in isolating complex conjugates because we know in advance that the exponent will be real-valued, and we also know from the form of $Q_\theta$ that exchanging $x$ and $\xi$ should exchange $\theta$ and $-\theta$.

We set $(y,\eta) = (x,\xi)$ in the above computation and simplify the exponent and the factor according to the notations just introduced. With $M_{-tq_\theta}$ as in example \ref{ex_NSAHO_Mehler} and with the sharp product defined by \eqref{eq_sharp_product}, the Weyl symbol of $(e^{-tQ_\theta})^* e^{-tQ_\theta}$ is therefore
\[
	(M_{-\bar{t}q_{-\theta}}\sharp M_{-tq_\theta})(x,\xi) = \frac{2}{|f(A_\theta)\sinh 2t|} e^{p(x,\xi)}
\]
with
\[
	p(x,\xi) = -e^{i\phi}(A_\theta x^2 + A_{-\theta}\xi^2) + \frac{e^{-i\phi}}{f(A_{-\theta})}(ix + e^{i\phi}A_\theta \xi)^2 + \frac{e^{-i\phi}}{f(A_\theta)}(i\xi - e^{i\phi}A_{-\theta}\xi)^2.
\]
The coefficient of $x^2$ is
\[
	\frac{1}{2}\partial_x^2 p(x,\xi) = -e^{i\phi}A_\theta - \frac{e^{-i\phi}}{f(A_{-\theta})} + \frac{e^{-i\phi}}{f(A_{\theta})}(e^{i\phi}A_\theta)^2 = -2\Re\left(\frac{e^{i\phi}}{f(A_\theta)}\right),
\]
using that $e^{i\phi}A_\theta f(A_\theta) = e^{i\phi}(1+A_\theta^2)$ to combine the first and third terms. Continuing with similar computations, we arrive at
\[
	p(x,\xi) = -2\Re\left(\frac{e^{i\phi}}{f(A_\theta)}\right)x^2 + 4\Im\left(\frac{A_\theta}{f(A_\theta)}\right)x\xi - 2\Re\left(\frac{e^{i\phi}}{f(A_{-\theta})}\right)\xi^2.
\]

We identify the harmonic oscillator equivalent to $p^w(x,D_x)$ following example \ref{ex_HO_reduce}. The fundamental matrix of $p$ is
\[
	F = 2 \left(\begin{array}{cc} \Im\left(\frac{A_\theta}{f(A_\theta)}\right) & -\Re\left(\frac{e^{i\phi}}{f(A_{-\theta})}\right) \\ \Re\left(\frac{e^{i\phi}}{f(A_\theta)}\right) & -\Im\left(\frac{A_\theta}{f(A_\theta)}\right)\end{array}\right).
\]
Expanding real and imaginary parts using complex conjugates,
\[
	\begin{aligned}
	\det F &= 4 \left(-\left(\Im\left(\frac{A_\theta}{f(A_\theta)}\right)\right)^2 + \Re\left(\frac{e^{i\phi}}{f(A_{-\theta})}\right)\Re\left(\frac{e^{i\phi}}{f(A_\theta)}\right)\right)
	\\ & =\left(\frac{A_\theta}{f(A_{\theta})} - \frac{A_{-\theta}}{f(A_{-\theta})}\right)^2 + \left(\frac{e^{i\phi}}{f(A_{-\theta})} + \frac{e^{-i\phi}}{f(A_\theta)}\right) \left(\frac{e^{i\phi}}{f(A_{\theta})} + \frac{e^{-i\phi}}{f(A_{-\theta})}\right)
	\\ &= \frac{1+A_\theta^2}{f(A_\theta)^2} + \frac{1+A_{-\theta}^2}{f(A_{-\theta})^2} + \frac{2\cos 2\phi - 2A_0^2}{f(A_{\theta})f(A_{-\theta})}
	\\ &= \frac{1}{f(A_{\theta})f(A_{-\theta})}\left(A_\theta f(A_{-\theta}) + A_{-\theta}f(A_{\theta}) + 2\cos 2\phi - 2A_0^2\right)
	\\ &= \frac{2\cos 2\theta + 2\cos 2\phi}{|f(A_\theta)|^2}.
	\end{aligned}
\]
We compute furthermore that
\[
	\begin{aligned}
	|f(A_\theta)|^2 &= \left| (|T| + |T|^{-1})\cos \theta + i(|T| - |T|^{-1})\sin \theta\right|^2
	\\ &= |T|^2 + |T|^{-2} + 2\cos 2\theta
	\\ &= |T|^2 + |T|^{-2} - 2\cos 2\phi + 2\cos 2\theta + 2\cos 2\phi
	\\ &= \left||T|e^{i\phi} - \frac{1}{|T|e^{i\phi}}\right|^2 + 2\cos 2\theta + 2\cos 2\phi.
	\end{aligned}
\]
Noting that $T = |T|e^{i\phi}$ and that $T - T^{-1} = -\frac{2}{\sinh 2t}$ and writing 
\[
	A = \frac{1}{2}|\sinh 2t|^2 (\cos 2\theta + \cos 2\phi) = \left|\frac{1}{2}f(A_\theta)\sinh 2t\right|^2-1
\]
as in \eqref{eq_def_Acste}, we obtain
\[
	\det F = \frac{A}{1+A}.
\]

Note that $t \in \Omega_\theta$ if and only if both $\Re t > 0$ and $|\phi| \leq \pi/2-\theta.$ We see that $\cos 2\theta + \cos 2\phi = 0$ if and only if $|\phi| = \pi/2-\theta$. In this case, $p(x,\xi)$ is symplectically equivalent to $-x^2$ and the coefficient $2|f(A_\theta)\sinh(2t)|^{-1}$ is 1, so the conclusions of theorems \ref{thm_NSAHO_norm} and \ref{thm_NSAHO_SVD} for $t \in \partial \Omega_\theta \cap \{\Re t > 0\}$ follow.

Otherwise, $t \in \opnm{int}\Omega_\theta$, so $A > 0$ and $\det F \in (0,1)$. (This also follows by corollary \ref{cor_positive_compact} and the fact that $e^{-tQ_\theta}$ is compact on $\opnm{int}\Omega_\theta$, shown in propositions \ref{prop_FBI_boundedness} and \ref{prop_bddness_same}.) The symbol of $(e^{-tQ_\theta})^*e^{-tQ_\theta}$ is therefore symplectically equivalent to
\begin{equation}\label{eq_NSHAO_symbol_detF}
	(1+A)^{-1/2}e^{-\sqrt{\det F}(x^2 + \xi^2)}.
\end{equation}
From proposition \ref{prop_HO_Mehler_all_times}, we obtain the norm \eqref{eq_NSAHO_norm} of $\|e^{-tQ_\theta}\|$ in \eqref{eq_NSAHO_norm} as
\[
	\|(e^{-tQ_\theta})^*e^{-tQ_\theta}\|^{1/2} = \left(\sqrt{1+A}(1+\sqrt{\det F})\right)^{-1/2} = \left(\sqrt{1+A} + \sqrt{A}\right)^{-1/2}.
\]
The unitary reduction of $\left((e^{-tQ_\theta})^*e^{-tQ_\theta}\right)^{1/2}$ follows similarly. Since 
\[
	\cosh \arctanh \sqrt{\det F} = (1-\det F)^{-1/2} = \sqrt{1+A},
\]
we see that \eqref{eq_NSHAO_symbol_detF} is already the Mehler formula for $e^{-\arctanh({\sqrt{\det F}})Q_0}$, so taking the square root amounts to dividing the exponent by 2. This completes the proof of theorems \ref{thm_NSAHO_norm} and \ref{thm_NSAHO_SVD}.
\end{proof}

\begin{remark}\label{rem_bdry_Omega}
The case $t \in \frac{i\pi}{2}\Bbb{Z}$ should be treated separately since $\phi = \arg \tanh t$ is not well-defined.  We recall that the eigenfunctions of $Q_\theta$ may be realized as
\begin{equation}\label{eq_uk}
	u_k(x) = h_k(e^{i\theta/2}x) \in \ker(Q_\theta - (1+2k)),
\end{equation}
for $h_k$ the Hermite functions. In particular, the $u_k$ inherit the parity of the Hermite functions, $u_k(-x) = (-1)^k u_k(x)$, and we recall that the span of the $u_k$ is dense in $L^2(\Bbb{R})$. In this way, if $t = i \pi j \in i \pi \Bbb{Z}$ then $e^{-tQ_\theta}u_k(x) = e^{i \pi j}e^{2\pi i j k}u_k(x) = (-1)^ju_k(x)$ which extends by continuity to all of $L^2(\Bbb{R})$. To complete the analysis of all $t \in \frac{i\pi}{2}\Bbb{Z}$ it suffices to treat the case where $t = i\pi/2$. In this case, 
\[
	e^{-tQ_\theta}u_k(x) = -i(-1)^ku_k(x) = -iu_k(-x),
\]
which also extends by continuity.
\end{remark}

\begin{remark}
We also mention, without going into detailed computations, that this formula agrees with three known results.  First, for all $t \in \Bbb{C}$ with $\Re t \geq 0$,
\[
	\|e^{-tQ_0}\| = e^{-\Re t},
\]
since $Q_0$ is self-adjoint and $\opnm{Spec}Q_0 = \{1+2k\::\: k \in \Bbb{N}\}$. Second, for $\phi \in \Bbb{R}$ fixed obeying $|\phi| \leq \pi/2 - |\theta|$, identifying the real part of the operator $e^{i\phi}Q_\theta$ gives that
\[
	\mathop{\lim_{t \to 0}}_{t > 0} \frac{1}{t}(\|e^{-te^{i\phi}Q_\theta}\| - 1) = -\sqrt{\cos(\phi + \theta) \cos(\phi - \theta)}.
\]
Finally, from the analysis of the return to equilibrium (for instance, \cite[Theorem 2.2]{Ottobre_Pavliotis_Pravda-Starov_2012} or \cite[Theorem 4.2]{Aleman_Viola_2014b}),
\[
	\lim_{\Re t \to \infty} e^{\Re t} \|e^{-tQ_\theta}\| = \frac{1}{\sqrt{\cos \theta}},
\]
which is the norm of the first spectral projection for $e^{-tQ_\theta}$ (see \cite{Davies_Kuijlaars_2004}).
\end{remark}

\section{Embeddings between Fock spaces}\label{sec_FBI}

\subsection{Setting and equivalence with the non-self-adjoint harmonic oscillator semigroup}

We recall that we are interested in embeddings 
\[
	\iota: H_{\Phi_1} \ni u \mapsto u \in H_{\Phi_2},
\]
between spaces defined in \eqref{eq_def_HPhi} and for quadratic weights $\Phi_1, \Phi_2:\Bbb{C} \to \Bbb{R}$ which are strictly plurisubharmonic as in \eqref{eq_stplsh}.

Before recalling the relationship between this problem and the semigroup associated with the non-self-adjoint harmonic oscillator, we establish that we may reduce $\Phi_1$ and $\Phi_2$ to model weights: in particular, we may assume that $\Phi_1(z) = \frac{1}{2}|z|^2$.

\begin{lemma}\label{lem_weight_reduction}
Let $\Phi_1, \Phi_2:\Bbb{C}\to\Bbb{R}$ be real-valued real-quadratic forms which are strictly plurisubharmonic as in \eqref{eq_stplsh}, and as in theorem \ref{thm_HPhi_embedding}, let
\[
	a = \frac{\partial_{\bar{z}}\partial_z \Phi_2}{\partial_{\bar{z}}\partial_z \Phi_1}, \quad b = \frac{|\partial_z^2(\Phi_2-\Phi_1)|}{\partial_{\bar{z}}\partial_z\Phi_1}.
\]
Define the auxiliary weights
\[
	\Phi_0(z) = \frac{1}{2}|z|^2
\]
and
\[
	\tilde{\Phi}(z) = \frac{1}{2}\left(a|z|^2 - b\Re(z^2)\right).
\]
Then the (possibly unbounded) embeddings 
\[
	\iota:H_{\Phi_1}\to H_{\Phi_2}
\]
and
\[
	\tilde{\iota}:H_{\Phi_0}\to H_{\tilde{\Phi}}
\]
are unitarily equivalent.
\end{lemma}

\begin{proof}
Because $\Phi_1$ and $\Phi_2$ are real-valued quadratic functions on $\Bbb{C}$, for $j=1,2$,
\[
	\Phi_j(z) = a_j|z|^2 + \Re (b_jz^2)
\]
for
\[
	a_j = \partial_{\bar z}\partial_z \Phi_j, \quad b_j = \partial_z^2 \Phi_j.
\]
It is straightforward to check that, for $r, s \in \Bbb{C}$ with $r \neq 0$, the maps
\begin{equation}\label{eq_unitary_C}
	\mathcal{C}_r : H_{\Phi(z)} \ni u(z) \mapsto ru(rz) \in H_{\Phi(rz)}
\end{equation}
and
\[
	\mathcal{W}_s : H_{\Phi(z)} \ni u(z) \mapsto u(z)e^{sz^2} \in H_{\Phi(z) + \Re(sz^2)}
\]
are unitary operators for any real-valued real-quadratic strictly plurisubharmonic weight $\Phi$.

For $\mu \in \Bbb{C}$ with $|\mu| = 1$ to be determined, set
\[
	\mathcal{U} = \mathcal{C}_{\mu/\sqrt{2a_1}}\mathcal{W}_{-b_1}.
\]
Then $\mathcal{U}$ is unitary from $H_{\Phi_1}$ to $H_{\Phi_0}$ as well as from $H_{\Phi_2}$ to $H_{\hat{\Phi}}$ where
\[
	\hat{\Phi}(z) = \Phi_2\left(\frac{\mu z}{\sqrt{2a_1}}\right) - \Re \left(b_1\frac{\mu^2 z^2}{2a_1}\right) = \frac{a_2}{2a_1}|z|^2 - \Re \left(\frac{(b_1-b_2)\mu^2}{2a_1}z^2\right).
\]
Letting $\mu$ be such that $\mu^2(b_1-b_2) = |b_1-b_2|$ gives $\hat{\Phi} = \tilde{\Phi}$ and
\[
	\tilde{\iota} = \mathcal{U}\iota \mathcal{U}^*.
\]
\end{proof}

We use an FBI--Bargmann transform (here, FBI stands for Fourier{--}Bros{--}Iagolnitzer) to relate the embedding $\iota$ and the semigroup $e^{-tQ_\theta}$ following \cite{Aleman_Viola_2014b}. We recall the essential facts in dimension one and with complex quadratic phase, making reference to \cite[Chapter 13]{Zworski_2012}; one could also refer to \cite[Sections 12.2, 12.3]{Sjostrand_LoR}, \cite{Martinez_2002}, \cite{Folland_1989}, or the original paper \cite{Bargmann_1961}.

Let
\[
	\varphi(z,x) = \frac{1}{2}\alpha z^2 + \beta zx + \frac{1}{2}\gamma x^2
\]
with $\beta \neq 0$ and $\Im \gamma > 0$. Then
\begin{equation}\label{eq_def_FBI}
	\mathcal{T}_\varphi f(z) = c_\varphi\int_{\Bbb{R}} e^{i\varphi(z,x)}f(x)\,dx,
\end{equation}
with
\[
	c_\varphi = \frac{\beta}{2^{1/2}\pi^{3/4}(\Im \gamma)^{1/4}},
\]
is unitary from $L^2(\Bbb{R})$ to $H_\Phi(\Bbb{C})$ when
\[
	\begin{aligned}
	\Phi(z) &= \sup_{x \in \Bbb{R}} (-\Im \varphi(z,x))
	\\ &= \frac{1}{4 \Im \gamma}|\beta z|^2 - \frac{1}{4\Im \gamma}\Re ((\beta z)^2) - \frac{1}{2}\Im(\alpha z^2).
	\end{aligned}
\]
Conjugation with $\mathcal{T}_\varphi$ allows us to compose symbols with the linear canonical transformation, which is now allowed to be complex, 
\[
	K = -\frac{1}{\beta}\left(\begin{array}{cc} \alpha & -1 \\ \beta^2-\alpha\gamma & \gamma\end{array}\right).
\]

The Weyl quantization on $H_\Phi$ is defined for symbols $a:\Lambda_\Phi\to\Bbb{C}$ with
\[
	\Lambda_\Phi = \left\{\left(z, -2i\partial_z\Phi(z)\right)\right\} = K^{-1}(\Bbb{R}^{2n}),
\]
via the formula (see \cite[Eq.~(13.4.5)]{Zworski_2012} or \cite[Eq.~(12.32)]{Sjostrand_LoR})
\[
	a^w_\Phi(z,D_z)u(z) = \frac{1}{2\pi} \iint_{\zeta = \frac{2}{i}(\partial_z\Phi)(\frac{z+w}{2})} e^{i(z-w)\cdot \zeta}a\left(\frac{z+w}{2}, \zeta\right)u(w)\,d\zeta \wedge dw.
\]
For the weak definition the assumption that $u, v \in \mathscr{S}(\Bbb{R})$ is replaced by the assumption that 
\begin{equation}\label{eq_FBI_Schwartz}
	u, v \in \{w \in H_\Phi \::\: z^k w \in H_\Phi,~~\forall k \in \Bbb{N}\} = \mathcal{T}_\varphi(\mathscr{S}(\Bbb{R})),
\end{equation}
described in, for instance, \cite[p.~142]{Sjostrand_LoR}.

The exact Egorov theorem \cite[Theorem 13.9]{Zworski_2012} (cf.\ \cite[Proposition~3.3.1]{Martinez_2002} or \cite[Eq.~(12.37)]{Sjostrand_LoR}) gives that 
\begin{equation}\label{eq_Egorov}
	\mathcal{T}_\varphi a^w(x,D_x)\mathcal{T}_\varphi^* = (a \circ K)_\Phi^w(z,D_z),
\end{equation}
We will only apply \eqref{eq_Egorov} to polynomial (quadratic) symbols where the usual formulas apply.

\begin{example}\label{ex_NSAHO_FBI}
(See also \cite[Example 2.6]{Viola_2013}.) For $\theta \in (-\pi/2, \pi/2)$ fixed, let
\[
	\varphi(z,x) = ie^{i\theta}\left(\frac{1}{2}(x^2 + z^2) - \sqrt{2}zx\right).
\]
This complex scaling of the classical phase \cite[Eq.~(2.1)]{Bargmann_1961} is chosen so that
\begin{equation}\label{eq_K_NSHAO_FBI}
	K = \frac{1}{\sqrt{2}}\left(\begin{array}{cc} 1 & ie^{i\theta} \\ ie^{-i\theta} & 1 \end{array}\right)
\end{equation}
which gives
\[
	(q_\theta\circ K)(x,\xi) = 2ix\xi.
\]
Therefore, for $\mathcal{T}$ as in \eqref{eq_def_FBI} with $\varphi$ above,
\begin{equation}\label{eq_NSAHO_Egorov}
	\mathcal{T} Q_\theta \mathcal{T}^* = 2z\frac{d}{dz} + 1.
\end{equation}

The corresponding weight is
\begin{equation}\label{eq_NSAHO_weight}
	\Phi(z) = \frac{1}{2\cos \theta}|z|^2 + \frac{1}{2\cos\theta} \Re(ie^{i\theta}(\sin\theta) z^2).
\end{equation}
\end{example}

This particular FBI--Bargmann transform allows us to establish a natural maximal definition of $e^{-tQ_\theta}$ for any $t \in \Bbb{C}$, as shown in \cite{Aleman_Viola_2014b}. As a core for this operator, we take 
\begin{equation}\label{eq_def_V_core}
	V = \opnm{Span}\{u_k\}_{k=0}^\infty, \quad u_k \in \opnm{ker}\left(Q_\theta - (1+2k)\right),
\end{equation}
for $\{u_k\}$ the eigenfunctions of $Q_\theta$ as in \eqref{eq_uk}. For $t \in \Bbb{C}$, let $S_t$ be the operator with domain $V$ defined by linear extension from
\begin{equation}\label{eq_semigroup_by_eigenvalues}
	S_t u_k = e^{-tQ_\theta}u_k = e^{-t(1+2k)}u_k, \quad k \in \Bbb{N}.
\end{equation}

\begin{proposition}\label{prop_NSAHO_FBI_side} 
Fix $\theta \in (-\pi/2, \pi/2)$ and let $Q_\theta$ be as in \eqref{eq_def_NSAHO}. Let the FBI--Bargmann transform $\mathcal{T}$ and the weight $\Phi$ be as in example \ref{ex_NSAHO_FBI}. For any $t\in\Bbb{C}$, let $S_t = e^{-tQ_\theta}$ with domain $V$ in \eqref{eq_def_V_core}.

Then $S_t$ is closable with closure
\begin{equation}\label{eq_NSAHO_semigroup_embedding_equiv}
	\overline{S_t} = e^{t}\mathcal{T}^* (\mathcal{C}_{e^{-2t}}) \iota \mathcal{T},
\end{equation}
where the embedding $\iota: H_{\Phi} \to H_{\Phi(e^{2t}z)}$ has the domain
\begin{equation}\label{eq_NSAHO_embedding_domain}
	\opnm{Dom}(\iota) = H_\Phi \cap H_{\Phi(e^{2t}z)}.
\end{equation}
and $\mathcal{C}_{e^{-2t}} : H_{\Phi(e^{2t}z)} \to H_{\Phi(z)}$ is as in \eqref{eq_unitary_C}. As a consequence,
\[
	\opnm{Dom}(\overline{S_t}) = \mathcal{T}^* \opnm{Dom}(\iota) = \{u \in L^2(\Bbb{R}) \::\: \mathcal{T}u \in H_{\Phi(e^{2t}z)}\}.
\]
\end{proposition}

\begin{proof}
Conjugation by $\mathcal{T}$ and the exact Egorov relation \eqref{eq_NSAHO_Egorov} reduces the Cauchy problem
\[
	\left\{\begin{array}{l}\partial_t U(t,z) + Q_\theta U(t,z) = 0, \\ U(0,z) = u(z) \in L^2(\Bbb{R}), \end{array}\right.
\]
to 
\[
	\left\{\begin{array}{l}\partial_t V(t,z) + \left(2z\partial_z + 1\right)V(t,z) = 0, \\ V(0,z) = v(z) = \mathcal{T}u(z) \in H_\Phi. \end{array}\right.
\]
This latter problem has the solution
\begin{equation}\label{eq_def_Rt}
	V(t,z) = R_t v(z) = e^{-t}v(e^{-2t}z)
\end{equation}
which is unique among holomorphic functions since any solution $V(t,z)$ satisfies $\partial_t e^t V(t, e^{2t}z) \equiv 0$ (cf.\ \cite[Proposition 2.1]{Aleman_Viola_2014b}). This agrees with the solution by eigenvalues \eqref{eq_semigroup_by_eigenvalues} on $\mathcal{T}(V)$ because $\mathcal{T}u_k = c_k z^k$ for some normalizing constants $c_k$ and
\[
	R_t z^k = e^{-t}(e^{-2t}z)^k = e^{-t(1+2k)}z^k.
\]
A deformation argument \cite[Proposition 2.8]{Aleman_Viola_2014b} shows that the closure of the restriction of $R_t$ to the polynomials $\Bbb{C}[z] = \mathcal{T}(V) \subset H_\Phi$ is $R_t$ equipped with its maximal domain $\{v \in H_\Phi \::\: R_t v \in H_\Phi\}$.

The conclusion \eqref{eq_NSAHO_semigroup_embedding_equiv} follows from the observation that 
\[
	R_t = e^{t}(\mathcal{C}_{e^{-2t}})\iota
\]
as (possibly unbounded) operators on $H_\Phi$, where $R_t$ has its maximal domain and $\iota$ is equipped with the domain \eqref{eq_NSAHO_embedding_domain}.
\end{proof}

\begin{remark}\label{rem_unbdd_Hormander}
The unboundedness of certain Mehler formulas related to non-self-adjoint harmonic oscillators was also mentioned in the final remark of \cite{Hormander_1995}. The works  \cite{Graefe_Schubert_2012, Graefe_Korsch_Rush_Schubert_2015, Lasser_Schubert_Troppmann_2015} also investigate the effect of $e^{itQ_\theta}, t \in \Bbb{R}$ on Gaussians and observe that the evolution may become unbounded.
\end{remark}

When $\Phi_1(z) = \Phi(z)$ and $\Phi_2(z) = \Phi(e^{2t}z)$ with $\Phi$ from \eqref{eq_NSAHO_weight}, the quantities $a$ and $b$ from theorem \ref{thm_HPhi_embedding} or lemma \ref{lem_weight_reduction} are
\begin{equation}\label{eq_NSAHO_ab}
	a = \frac{\partial_{\bar{z}}\partial_z \Phi_2}{\partial_{\bar{z}}\partial_z \Phi_1} = e^{4\Re t}, \quad b = \frac{|\partial_z^2(\Phi_2-\Phi_1)|}{\partial_{\bar{z}}\partial_z\Phi_1} = |(e^{4t}-1)\sin\theta|.
\end{equation}
From \cite[Theorem 2.9]{Aleman_Viola_2014b}, we obtain the following characterization of boundedness and compactness of this definition of $e^{-tQ_\theta}$.

\begin{proposition}\label{prop_FBI_boundedness}
For $\theta \in (-\pi/2, \pi/2)$ and $t \in \Bbb{C}$, let $\Phi$ be as in \eqref{eq_NSAHO_weight}, and let $a = a(t,\theta)$ and $b = b(t,\theta)$ be as in \eqref{eq_NSAHO_ab} above. The operator $e^{-tQ_\theta} := \overline{S_t}$ from proposition \ref{prop_NSAHO_FBI_side} is bounded if and only if $t \in \Omega'_\theta$ where
\[
	\Omega'_\theta = \{t\::\: \Phi(z) \leq \Phi(e^{2t}z),~~\forall z \in \Bbb{C}\} = \{t\::\: a-b \geq 1\},
\]
on which the family $e^{-tQ_\theta}$ is strongly continuous in $t$, and is compact if and only if $t \in \Omega''_\theta$ where
\[
	\Omega''_\theta = \{t\::\: \Phi(z) < \Phi(e^{2t}z),~~\forall z \in \Bbb{C} \backslash\{0\}\} = \{t\::\: a-b > 1\}.
\]
\end{proposition}

It is worth noting that this realization of $e^{-tQ_\theta}$ gives a strong solution when $t \in \Omega''_\theta$ for every $u\in L^2(\Bbb{R})$, and for certain sufficiently regular and rapidly decaying functions for other $t$.

\begin{proposition}
For $\theta \in (-\pi/2, \pi/2)$, let $Q_\theta$ be as in \eqref{eq_def_NSAHO}, and let $\mathcal{T}$ and $\Phi$ be as in example \ref{ex_NSAHO_FBI}. Suppose that $t_0 \in \Bbb{C}$ and $u \in L^2(\Bbb{R})$ are such that 
\begin{equation}\label{eq_strongcty_cond}
	\exists \eps > 0 \::\: \mathcal{T}u \in H_{\Phi(e^{2t_0}z)-\eps|z|^2}.
\end{equation}
In particular, whenever $t_0 \in \Omega''_\theta$, this holds for all $u \in L^2(\Bbb{R})$, and whenever $u \in V$ with $V$ from \eqref{eq_def_V_core}, this holds for all $t_0 \in \Bbb{C}$.

Then for every $t$ in a sufficiently small complex neighborhood of $t_0$, $e^{-tQ_\theta}u(x) := \overline{S_t}u(x) \in \mathscr{S}(\Bbb{R})$ is a Schwartz function in $x$, an analytic function in $t$, and a strong solution of $(\partial_t + Q_\theta)\overline{S_t}u = 0$.
\end{proposition}

\begin{proof}
The fact that \eqref{eq_strongcty_cond} implies that $\overline{S_{t_0}}u \in \mathscr{S}(\Bbb{R})$ follows from the condition in \eqref{eq_FBI_Schwartz}, since by \eqref{eq_NSAHO_semigroup_embedding_equiv} and unitarity of $\mathcal{C}_{e^{-2t}}:H_{\Phi(e^{2t}z)} \to H_{\Phi}$, we may compute
\[
	\begin{aligned}
		\|z^k \mathcal{T}\overline{S_{t_0}}u\|_\Phi^2 &= \|z^k e^{t_0}(\mathcal{C}_{e^{-2{t_0}}})\iota \mathcal{T}u\|_\Phi^2
		\\ &= \|e^{t_0} (\mathcal{C}_{e^{-2{t_0}}})\iota(e^{2{t_0}}z)^k \mathcal{T}u\|_\Phi^2
		\\ &= \|e^{(1+2k){t_0}}z^k\mathcal{T}u\|_{\Phi(e^{2{t_0}}z)}^2
		\\ &= e^{(2+4k){t_0}}\int |z|^{2k}e^{-2\eps|z|^2}|u(z)|^2 e^{-2(\Phi(e^{2t_0}z) - \eps|z|^2)}\,d\Re z\,d\Im z.
	\end{aligned}
\]
This last quantity is finite, since $|z|^{2k}e^{-2\eps|z|^2}$ is bounded on $\Bbb{C}$ for any $k \in \Bbb{N}$ fixed and since we are working under the assumption \eqref{eq_strongcty_cond}. Furthermore, replacing $\eps$ by $\eps/2$ allows us to establish \eqref{eq_strongcty_cond} in a neighborhood of $t_0$. Then analyticity in $t$ follows from differentiating $R_t\mathcal{T}u(z)$ from \eqref{eq_def_Rt}, and $(\partial_t + Q_\theta)\overline{S_t}u = 0$ by the exact Egorov relation \eqref{eq_NSAHO_Egorov}.

The condition \eqref{eq_strongcty_cond} holds when $t_0 \in \Omega''_\theta$ for any $u \in L^2(\Bbb{R})$ since $\mathcal{T}u \in H_\Phi$ and strict convexity implies the existence of some $\eps > 0$ such that
\[
	\Phi(e^{2t_0}z) - \Phi(z) \geq \eps |z|^2.
\]
The condition holds when $u \in V$ for any $t_0 \in \Bbb{C}$ because $\Phi(e^{2t_0}z)$ is always strictly convex and $\mathcal{T}u$ is a polynomial.
\end{proof}

Having identified $\Omega'_\theta$, the maximal set of $t \in \Bbb{C}$ where $e^{-tQ_\theta}$ can be bounded as a closed operator whose domain includes $V$, we check that this is identical to $\overline{\Omega_\theta}$ in \eqref{eq_def_Omega_theta}, the set of $t \in \Bbb{C}$ for which the Mehler formula $M_{-tq_\theta}$ in example \ref{ex_NSAHO_Mehler} is bounded as a function of $(x,\xi) \in \Bbb{R}^2$.

\begin{proposition}\label{prop_bddness_same}
For $\theta \in (-\pi/2, \pi/2)$, let $\Omega_\theta$ be as in \eqref{eq_def_Omega_theta} and let $\Omega'_\theta$ and $\Omega''_\theta$ be as in proposition \ref{prop_FBI_boundedness}. Then
\[
	\Omega'_\theta = \overline{\Omega_\theta}
\]
and
\[
	\Omega''_\theta = \opnm{int} \Omega_\theta.
\]
\end{proposition}

\begin{proof}
We check the first equality; the second follows from making the inequalities strict. 

Let $a$ and $b$ be as in \eqref{eq_NSAHO_ab}. When $\theta = 0$, the condition $a-b \geq 1$ reduces to $e^{4\Re t} \geq 1$, which agrees with $\overline{\Omega_0} = \{\Re t \geq 0\}$. When $\theta \neq 0$ and $\Re t = 0$, it is easy to see that $a-b \geq 1$ if and only if $t \in \frac{i\pi}{2}\Bbb{Z}$. We will therefore suppose that $\theta \neq 0$ and $\Re t > 0$.

So long as $\Re t > 0$, the condition $a-b \geq 1$ with $a,b$ from \eqref{eq_NSAHO_ab} is equivalent to
\begin{equation}\label{eq_bdd_condition_FBI}
	\frac{e^{4\Re t} - 1}{|e^{4t}-1|} = \frac{\sinh 2\Re t}{|\sinh 2t|} \geq \sin|\theta|.
\end{equation}
Taking the square of the reciprocal (since we are assuming that $\theta \neq 0$ and $\Re t > 0$), we see that $a-b \geq 1$ is equivalent to
\begin{equation}
	\frac{|e^{4t}-1|^2}{(e^{4\Re t} - 1)^2} - 1 \leq \cot^2 |\theta|.
\end{equation}

On the other hand, note that so long as $\Re t \neq 0$,
\[
	\arg \tanh t = \arg\left((e^t - e^{-t})(e^{\bar{t}} + e^{-\bar{t}})\right) = \arctan \left(\frac{\sin 2\Im t}{\sinh 2\Re t}\right).
\]
Therefore, when $\theta \neq 0$ and $|\theta| < \pi/2$, the condition $t \in \Omega_\theta$ defined in \eqref{eq_def_Omega_theta} is equivalent to assuming that $\Re t > 0$ and that
\begin{equation}\label{eq_bdd_condition_real}
	\frac{(\sin 2\Im t)^2}{(\sinh 2 \Re t)^2} \leq \cot^2|\theta|.
\end{equation}

Equivalence of \eqref{eq_bdd_condition_FBI} and \eqref{eq_bdd_condition_real} follows from the computation
\begin{multline*}
	\frac{|e^{4t}-1|^2}{(e^{4\Re t} - 1)^2} - 1 = \frac{|e^{2t}|^2 |\sinh 2t|^2}{e^{4\Re t}(\sinh 2\Re t)^2} - 1
	\\ = \frac{|\sinh 2t|^2 - (\sinh 2\Re t)^2}{(\sinh 2\Re t)^2} = \frac{(\sin 2\Im t)^2}{(\sinh 2\Re t)^2},
\end{multline*}
which in turn follows from
\begin{multline*}
	|\sinh 2t|^2 = \frac{1}{4}(e^{2t}-e^{-2t})(e^{2\bar{t}} - e^{-2\bar{t}}) 
	\\= \frac{1}{4}\left(2\cosh 4\Re t - 1 - 1 + 2\cos 4\Im t\right) = (\sinh 2\Re t)^2 + (\sin 2\Im t)^2.
\end{multline*}
\end{proof}

\subsection{Computation of the norm}\label{ssec_norm_FBI}

Using lemma \ref{lem_weight_reduction}, we assume in what follows that 
\[
	\Phi_1(z) = \frac{1}{2}|z|^2,
\]
\[
	\Phi_2(z) = \frac{1}{2}\left(a|z|^2 - b\Re z^2\right)
\]
with $b \geq 0$ and $a-b \geq 1$. We can realize any $a, b$ satisfying these conditions via \eqref{eq_NSAHO_ab} by taking
\[
	t = \frac{i\pi}{4} + \frac{1}{4}\log a, \quad \theta = \arcsin \frac{b}{a+1},
\]
so every such embedding is unitarily equivalent (via metaplectic and FBI--Bargmann transformations) to some non-self-adjoint harmonic oscillator semigroup. Therefore, by proposition \ref{prop_NSAHO_FBI_side} and theorem \ref{thm_NSAHO_SVD}, we know in advance that the operator $\left((\iota)^*\iota\right)^{1/2}$ is of harmonic oscillator or heat semigroup type, meaning that it is unitarily equivalent --- but with an operator corresponding to a complex canonical transformation --- to a constant multiple of a semigroup generated by a harmonic oscillator or, in the borderline case, multiplication by a Gaussian. Following proposition \ref{prop_gaussians}, we may obtain the norm as the supremum over Gaussians $u_\gamma$ defined in \eqref{eq_u_gamma_def}.

The FBI--Bargmann transforms preserve Gaussians, in the sense of kernels of linear forms in $(x,D_x)$ described in \cite[Section 5]{Hormander_1995} and in the proof of proposition \ref{prop_gaussians}. We use the classical Bargmann transform $\mathcal{T}$, which may be taken from example \ref{ex_NSAHO_FBI} with $\theta = 0$; note that $\mathcal{T}:L^2(\Bbb{R}) \to H_{\Phi_1}$. The canonical transformation \eqref{eq_K_NSHAO_FBI} with $\theta = 0$ induces by \eqref{eq_lft_gaussians} the Cayley transform
\[
	\Bbb{C}_+ \ni \gamma \mapsto L(\gamma) = \frac{\gamma-1}{\gamma+1} \in \{|\mu| < 1\}
\]
instead of an automorphism of $\Bbb{C}_+$. Therefore the set of integrable Gaussians is transformed from 
\[
	\{u_\gamma(x) \,:\, \gamma \in \Bbb{C}_+\} \subset L^2(\Bbb{R})
\]
to
\[
	\{u_\gamma(z) \,:\, |\gamma| < 1\} \subset H_{\Phi_1}.
\]
(The normalization factor in $H_{\Phi_1}$ is different from the one in $L^2(\Bbb{R})$, but this does not meaningfully change the analysis which follows.) The boundary $\{u_\gamma(x) \,:\, \gamma \in i\Bbb{R} \cup \{+\infty\}\}$, with the convention $u_{+\infty} = \delta_0$, is replaced by $\{u_\gamma(z)\,:\, |\gamma| = 1\}$. We remark that the correspondance $L(\infty) = 1$ agrees with the direct computation that $\mathcal{T}\delta_0 (z) = c_1 e^{-z^2/2} = c_2 u_{1}(z)$ for some constants $c_1$ and $c_2$.

This reasoning allows us to conclude that, instead of finding $\|\iota\|$ by maximizing over all $u\in H_{\Phi_1}$, it suffices to maximize over the set of Gaussians in $H_{\Phi_1}$:
\[
	\|\iota\| = \sup_{|\gamma| < 1} \frac{\|u_\gamma\|_{\Phi_2}}{\|u_\gamma\|_{\Phi_1}}.
\]
Maximizing this quantity is an elementary exercise which we detail below. 

In view of theorem \ref{thm_higher_dimension}, this strategy could also be used to obtain the norm of any operator $e^{-tQ}$, when this operator is compact, with $Q$ quadratic and supersymmetric as in \cite{Aleman_Viola_2014b}.

Whenever $|\gamma| < 1$, we find, by changing variables to replace $\Re(\gamma z^2)$ by $|\gamma|\Re(z^2)$, that
\[
	\|u_\gamma\|_{\Phi_1}^2 = \frac {\sqrt{\pi\Re \gamma}}{\sqrt{1-|\gamma|^2}}.
\]
By a similar computation, if $|\gamma| < 1$ and $a-b \geq 1$, then
\[
	\frac{\|u_\gamma\|_{\Phi_2}}{\|u_\gamma\|_{\Phi_1}} = \left(\frac{1-|\gamma|^2}{a^2 - |b-\gamma|^2}\right)^{1/4}.
\]
Naturally, we maximize the fourth power of this quantity.

The partial derivative with respect to $\Im \gamma$ is 
\[
	\frac{\partial}{\partial \Im \gamma}\left(\frac{1-|\gamma|^2}{a^2 - |b-\gamma|^2}\right) = \frac{2(\Im \gamma)(1-a^2-2b\Re \gamma + b^2)}{(a^2 - |b-\gamma|^2)^2},
\]
and using $a-b \geq 1$ and $|\gamma| < 1$,
\[
	\begin{aligned}
	1-a^2-2b\Re \gamma + b^2 & = 1-(a-b)(a+b) - 2b\Re \gamma
	\\ &\leq 1-(a+b)-2b\Re \gamma
	\\ &\leq 1-a-b-2b
	\\ &= 1-a+b \leq 0.
	\end{aligned}
\]
Furthermore, apart from the trivial case $(a,b) = (1,0)$, one of the two inequalities must be strict. The derivative therefore has the same sign as $-\Im \gamma$, giving a global maximum when $\Im \gamma = 0$.

If $b = 0$, then the same reasoning gives that the maximum is attained at $\gamma = 0$, from which we deduce that $\|\iota\| = a^{-1/2}$ when $b = 0$.

Assuming, as we therefore may, that $\gamma \in \Bbb{R}$ and that $b \neq 0$,
\[
	\frac{d}{d\gamma}\left(\frac{1-\gamma^2}{a^2 - (b-\gamma)^2}\right) = -2\frac{b\gamma^2 + (a^2-b^2-1)\gamma + b}{(a^2 - (\gamma - b)^2)^2}.
\]
If $a-b > 1$, the maximum is therefore at the root of the numerator
\[
	\gamma = \frac{1}{2b}\left(1 - a^2 + b^2 + \sqrt{(a^2 - b^2 - 1)^2 - 4b^2}\right),
\]
as in theorem \ref{thm_HPhi_embedding}. (The condition $a-b > 1$ implies that this is the only root satisfying $|\gamma| \leq 1$.)
When $a-b = 1$, we obtain the same result in the limit $\gamma \to -1^+$. This proves theorem \ref{thm_HPhi_embedding}.

\section{Extensions to any dimension}\label{sec_higher_dimension}

If $Q = q^w(x,D_x)$ with $q:\Bbb{R}^{2n}\to\Bbb{C}$ satisfying $\Re q(x,\xi) \geq \frac{1}{C}|(x,\xi)|^2$, the Mehler formula \eqref{eq_Mehler} gives that the Weyl symbol of $e^{-tQ}$ is given by 
\[
	\opnm{Symb}(e^{-tQ}) = c(t,q)\exp\left(-\frac{1}{2}(x,\xi)\cdot A(t,q) (x,\xi)\right)
\]
where $c(t,q) \in \Bbb{C}$ and $A(t,q)$ is a symmetric matrix. By \cite[Theorem 4.2]{Hormander_1995}, the real part $\Re A$ is positive definite. To find the operator norm of $e^{-tQ}$, it suffices to find the operator norm of the Weyl quantization of the exponential, and we begin by computing the sharp product of two such symbols. This result may also be found in, for instance, \cite[Theorem (5.6)]{Folland_1989}. We state the result only for Gaussian symbols which tend to zero as $|x|+|\xi| \to \infty$ in order to avoid the significant and interesting complications discussed in \cite[pp.\ 427-436]{Hormander_1995}.

\begin{proposition}\label{prop_sharp_higher_dim}
	Let $A_1$ and $A_2$ be symmetric $2n \times 2n$ matrices with positive definite real parts, and for $j = 1,2$, let
	\[
		a_j(x,\xi) = \exp\left(-\frac{1}{2}(x,\xi)\cdot A_j (x,\xi)\right).
	\]
	With $J$ as in \eqref{eq_def_J}, define
	\begin{equation}\label{eq_def_B_D}
		\begin{aligned}
		D &= 1 - \frac{1}{4}A_2JA_1J,
		\\ B &= A_1 + \left(1+\frac{i}{2}A_1J\right)D^{-1}A_2\left(1-\frac{i}{2}JA_1\right).
		\end{aligned}
	\end{equation}
	Then
	\[
		a_1 \sharp a_2(x,\xi) = \frac{1}{\sqrt{\det D}}\exp\left(-\frac{1}{2}(x,\xi) \cdot B (x,\xi)\right).
	\]	
\end{proposition}

\begin{proof}
	Taking the Fourier transform of a Gaussian in $Z_* = (X_*, Y_*) = (x_*, \xi_*, y_*, \eta_*)$ gives that
	\begin{multline*}
	e^{\frac{i}{2}\sigma((D_x, D_\xi), (D_y, D_\eta))}a_1(x,\xi)a_2(y,\eta)
	\\ = (2\pi)^{-4n} \int e^{iZ^*(Z-Z_*)+\frac{i}{2}(\xi^*y^*-\eta^*x^*) - \frac{1}{2}X_*\cdot A_1X_* - \frac{1}{2}Y_*\cdot A_2Y_*}\,dZ_*dZ^*
	\\ = \frac{(2\pi)^{-2n}}{\sqrt{\det(A_1A_2)}}\int e^{iZ^*\cdot Z + \frac{i}{2}(x^*, \xi^*)\cdot J(y^*, \eta^*) - \frac{1}{2}X^*\cdot A_1^{-1}X^* - \frac{1}{2}Y^*\cdot A_2^{-1}Y^*} \,dZ^*.
	\end{multline*}
	The exponent is $iZ^*\cdot Z-\frac{1}{2}Z^* \cdot \tilde{A}Z^*$ where
	\[
		\tilde{A} = \left(\begin{array}{cc} A_1^{-1} & -\frac{i}{2}J \\ \frac{i}{2}J & A_2^{-1}\end{array}\right).
	\]
	Again using the formula for the Fourier transform of a Gaussian, we obtain that
	\[
		e^{\frac{i}{2}\sigma((D_x, D_\xi), (D_y, D_\eta))}a_1(x,\xi)a_2(y,\eta) = \frac{1}{\sqrt{\det A_1A_2\tilde{A}}}\exp\left(-\frac{1}{2}Z\cdot \tilde{A}^{-1}Z\right).
	\]

	We note that $D$ defined in \eqref{eq_def_B_D} is invertible because
	\[
		\langle A_2^{-1}X, DX\rangle = \langle X, A_2^{-1}X\rangle + \frac{1}{4}\langle JX, A_1JX\rangle
	\]
	has positive real part whenever $X \neq 0$. We perform row reduction to find $\tilde{A}^{-1}$ as follows:
	\[
		\left(\begin{array}{cc}1 & \frac{i}{2}A_1J \\ 0 & 1\end{array}\right)\left(\begin{array}{cc} 1 & 0 \\ 0 & D^{-1}A_2\end{array}\right) \left(\begin{array}{cc}1 & 0 \\ -\frac{i}{2}J & 1\end{array}\right)\left(\begin{array}{cc}A_1 & 0 \\ 0 & 1\end{array}\right)\tilde{A} = 1.
	\]
	Observing that $\det \tilde{A}^{-1} = \det A_1DA_2$ and expanding $(X,X)\cdot\tilde{A}^{-1}(X,X)$ gives the proposition.
\end{proof}

This allows us to compute the norm in $\mathcal{L}(L^2(\Bbb{R}^{n}))$ of the Weyl quantization of any integrable Gaussian on $\Bbb{R}^{2n}$.

\begin{theorem}\label{thm_higher_dimension}
	Let $a(x,\xi):\Bbb{R}^{2n} \to \Bbb{C}$ be given by
	\[
		a(x,\xi) = \exp\left(-\frac{1}{2}(x,\xi)\cdot A (x,\xi)\right),
	\]
	for $A$ a symmetric $2n\times 2n$ matrix with $\Re A$ positive definite, and let $a^w(x,D_x)$ be the Weyl quantization of $a(x,\xi)$ as in \eqref{eq_def_Weyl}. With $J$ as in \eqref{eq_def_J}, define
	\[
		D = 1 - \frac{1}{4}AJA^*J
	\]
	and write
	\[
		B = A^* + \left(1+\frac{i}{2}A^*J\right)D^{-1}A\left(1-\frac{i}{2}JA^*\right).
	\]
	Then, as an operator on $L^2(\Bbb{R}^n)$,
	\[
		\|a^w(x,D_x)\| = (\det D)^{-1/4}\mathop{\prod_{2\mu \in \opnm{Spec}(-JB)}}_{\Im \mu > 0} (1 -i\mu)^{-1/2},
	\]
	with $\mu$, the eigenvalues of $-\frac{1}{2}JB$, in the product repeated according to their algebraic multiplicity.
\end{theorem}

\begin{proof}
	From proposition \ref{prop_sharp_higher_dim}, the symbol of $(a^w(x,D_x))^*a^w(x,D_x)$ is
	\[
		a(x,\xi) \sharp a(x,\xi) = \frac{1}{\sqrt{\det D}}\exp\left(-\frac{1}{2}(x,\xi) \cdot B (x,\xi)\right).
	\]
	Since $(a^w)^*a^w$ is a bounded self-adjoint compact operator on $L^2(\Bbb{R}^n)$, the matrix $B$ is real symmetric and positive definite. Therefore the exponent is symplectically equivalent to some harmonic oscillator symbol, and since the eigenvalues of the fundamental matrix are symplectic invariants, there is a unitary equivalence via an element of the metaplectic group
	\[
		-\frac{1}{2}(x,\xi) \cdot B (x,\xi) \sim -\sum_{j=1}^{n} \frac{\mu_j}{i}(x_j^2 + \xi_j^2)
	\]
	for $\mu_j$, with $j = 1,\dots, n$, those eigenvalues of $\frac{1}{2}JB$ for which $\mu_j/i > 0$, repeated for multiplicity.

	Therefore, up to unitary equivalence, $(a^w)^* a^w$ can be written as a tensor product of harmonic oscillator semigroups. By proposition \ref{prop_HO_Mehler_all_times} and positivity of $(a^w)^*a^w$, we see that $\mu_j/i \in (0,1)$ for all $j$. The conclusion of theorem \ref{thm_higher_dimension} follows from proposition \ref{prop_HO_Mehler_all_times}.
\end{proof}

By \cite[Theorem 4.2]{Hormander_1995}, when $q$ is elliptic and $t > 0$ the Weyl symbol of $e^{-tq^w(x,D_x)}$ is a constant times an integrable Gaussian, and we may therefore compute the corresponding operator norm.

\begin{corollary}\label{cor_higher_dim_Mehler}
	Let $Q = q^w(x,D_x)$ for $q:\Bbb{R}^{2n} \to \Bbb{C}$ quadratic and elliptic in the sense that $\Re q(x,\xi) > 0$ for all $(x,\xi) \in \Bbb{R}^{2n} \backslash \{0\}$. Then, for all $t > 0$, the norm of $e^{-tQ}$ as an operator in $\mathcal{L}(L^2(\Bbb{R}^{n}))$ may be computed by applying theorem \ref{thm_higher_dimension} to the Mehler formula \eqref{eq_Mehler}.
\end{corollary}

\begin{remark}\label{rem_higher_dim}
The formula obtained from theorem \ref{thm_higher_dimension} does not immediately give a simple expression for the norm: for instance, theorem \ref{thm_NSAHO_norm} does not automatically reveal certain simple facts like that $\|e^{-tQ_0}\| = e^{-\Re t}$ whenever $\Re t \geq 0$.

For another example, we recall the class of operators considered in \cite{Aleman_Viola_2014a}, with symbols
\[
	q(x,\xi) = \frac{1}{2}M(\xi + ix)\cdot (\xi - ix), \quad (x,\xi) \in \Bbb{R}^{2n},
\]
for $M$ an $n$-by-$n$ matrix. Writing $D_x = -i\nabla_x$, the corresponding Weyl quantization is
\[
	Q(x,D_x) = \frac{1}{2}M(D_x + ix) \cdot (D_x - ix) + \frac{1}{2}\opnm{Tr} M.
\]
The operator $e^{-tQ}$, realized in a way similar to that of proposition \ref{prop_NSAHO_FBI_side}, is bounded on $L^2(\Bbb{R}^n)$ if and only if $\|e^{-tM}\| \leq 1$, in which case 
\begin{equation}\label{eq_norm_M}
	\|e^{-tQ}\| = e^{-\frac{1}{2}\Re(t\opnm{Tr}M)}.
\end{equation}
(This is the case $N = 0$ of \cite[Corollary 2.9]{Aleman_Viola_2014a}.) It seems possible that one can deduce this result via theorem \ref{thm_higher_dimension}, but it certainly does not seem easy.

Nearly any such example shows that the boundary of the set of $t \in \Bbb{C}$ where $e^{-tQ}$ is bounded does not coincide with those $t \in \Bbb{C}$ for which $\|e^{-tQ}\| = 1$, even though this is the case for the non-self-adjoint harmonic oscillator. An important and interesting example is the Fokker-Planck quadratic model \cite[Section 5.5]{Helffer_Nier_2005}, given by
\[
	M = \left(\begin{array}{cc} 0 & -a \\ a & 1\end{array}\right), \quad a \in \Bbb{R}\backslash \{0\}.
\]
For this operator, the study of the return to equilbrium and the boundary of the set where $e^{-tQ}$ is bounded can be quite complicated; see, for instance, \cite[Section 1.2.2]{Aleman_Viola_2014b}. As soon as the operator is bounded, however, the computation of the operator norm is trivial by \eqref{eq_norm_M}. In particular, $\|e^{-tQ}\| = 1$ only if $e^{-tQ}$ is bounded and $\{\Re t = 0\}$, which only very rarely concides with the boundary of the set where $e^{-tQ}$ is bounded.
\end{remark}

\bibliographystyle{abbrv}
\bibliography{BibMehler}

\end{document}